\documentclass[10pt]{amsart} 
\usepackage{amsmath,amssymb}
\usepackage{graphicx}
\usepackage{float}
\usepackage{multicol}
\usepackage{ascmac}
\usepackage{mathtools}
\usepackage{bm}
\usepackage{amscd}
\usepackage{comment}
\usepackage{mathrsfs}
\usepackage[all]{xy}
\usepackage{color}
\usepackage{shuffle}

\usepackage{cases}
\usepackage{paralist}
\usepackage{hyperref}

\definecolor{Mycolor}{cmyk}{ 1.00, 0.50, 0.00, 0.0}  
\hypersetup{%
 setpagesize=false,%
 bookmarks=true,%
 bookmarksdepth=tocdepth,%
 bookmarksnumbered=true,%
 colorlinks=true,
 linkcolor=Mycolor,   
 citecolor=blue,   
 pdftitle={},%
 pdfsubject={},%
 pdfauthor={},%
 pdfkeywords={}}

\title{On adjoint double shuffle relations and adjoint reduction relations}
\author{Takumi Anzawa}
\email{atakumi86@gmail.com}
\address{
Academic Support Center, Kogakuin University,
2665-1, Nakano-cho, Hachioji,
Tokyo, 192-0015, Japan
}
\date{\today} 
\usepackage[margin=35mm]{geometry}
\theoremstyle{definition}
\newtheorem{df}{Definition}[section]

\newtheorem{prop}[df]{Proposition}
\newtheorem{lem}[df]{Lemma}
\newtheorem{cor}[df]{Corollary}
{\theoremstyle{remark}
\newtheorem{rmk}[df]{Remark}}
\newtheorem{thm}[df]{Theorem}

\newtheorem{qu}[df]{$ \pmb{Question}$}
\newtheorem{mt}{Main\;\;Theorem}

\DeclareMathOperator{\id}{id}

\DeclareMathOperator{\wt}{wt}

\newcommand{\ide}[1]{\mathfrak{#1}}

\newcommand{\del}{\partial}

\newcommand{\cotimes}{\widehat{\otimes}}
\newcommand{\cast}{\circledast}
\newcommand{\ad}{\mathfrak{ad}}
\newcommand{\Ad}{\mathrm{Ad}}
\newcommand{\Fad}{\mathfrak{F}_2^{\ad(x_1)}}
\newcommand{\Fadred}{\mathfrak{F}_2^{\ad(x_1), \red}}
\newcommand{\FAd}{\mathfrak{F}_2^{\Ad(x_1)}}
\newcommand{\cFad}{\widehat{\mathfrak{F}}_2^{\ad(x_1)}}
\newcommand{\cFadred}{\widehat{\mathfrak{F}}_2^{\ad(x_1),\red}}

\newcommand{\red}{\mathrm{red}}
\newcommand{\corr}{\mathrm{corr}}

\newcommand{\engMonth}{
	\ifcase\month
		\or January 
		\or February 
		\or March 
		\or April
		\or May 
		\or June
		\or July 
		\or August 
		\or September
		\or October 
		\or November 
		\or December\fi
} 
\newcommand{\fullday}{
	\engMonth \, \the\day, \the\year	
} 

\newcommand{\Z}{\mathbb{Z}}
\newcommand{\Q}{\mathbb{Q}}

\newcommand{\dmr}{\mathfrak{dmr}}
\newcommand{\addmr}{\mathfrak{addmr}}
\let\F\relax
\newcommand{\F}{\mathfrak{F}}

\newcommand{\DMR}{\mathrm{DMR}}
\newcommand{\AdDMR}{\mathrm{AdDMR}}

\usepackage{lineno}

\numberwithin{equation}{section}

\allowdisplaybreaks[1]
\begin{document}
\maketitle
\begin{abstract}
In this paper, we give a positive answer to Jarossay's question under the adjoint reduction relations.
Namely, we prove that the extended double shuffle relations together with the relation $\!\!\!\mod \pi^2$ are equivalent to the adjoint double shuffle relations together with the adjoint reduction relations.

\end{abstract}

\tableofcontents

\section{Introduction}

The \emph{extended double shuffle relations} are $\Q$-algebraic relations involving two bilinear operations, called the shuffle product and the harmonic product, which were formulated independently in \cite{Ihara-Kaneko-Zagier} and \cite{Racinet}. They are known as one of the main families of relations satisfied by multiple zeta values, and all algebraic relations among multiple zeta values over $\Q$ are expected to follow from the extended double shuffle relations. In \cite{Racinet}, Racinet defined an affine scheme $\DMR$ whose $R$-points are given by imposing the extended double shuffle relations over each commutative associative $\Q$-algebra $R$. In particular, its closed subscheme $\DMR_0$ is an affine group scheme, and $\DMR$ is a $\DMR_0$-torsor.

The \emph{adjoint double shuffle relations} are an \emph{adjoint analogue} of the extended double shuffle relations described above, and are known to be satisfied by adjoint multiple zeta values. The study of these relations is motivated by the Kaneko--Zagier conjecture \cite{Kaneko-Zagier} and the unpublished Akagi--Hirose--Yasuda conjecture. The Kaneko--Zagier conjecture predicts that two variants of multiple zeta values, namely \emph{finite multiple zeta values} and \emph{symmetric multiple zeta values}, satisfy the same algebraic relations over $\Q$. Symmetric multiple zeta values are constructed from certain polynomial expressions in multiple zeta values, and they span the unital $\Q$-algebra generated by multiple zeta values (\cite{Yasuda}). On the other hand, finite multiple zeta values are obtained by reducing certain partial sums of the infinite series defining multiple zeta values modulo primes.
As symmetric multiple zeta values are just expressed in terms of multiple zeta values, Akagi, Hirose, and Yasuda conjectured that finite multiple zeta values could be described by explicit formulas involving $p$-adic multiple zeta values. This conjecture was subsequently proved by Jarossay (\cite{Jarossay2020}). The adjoint operation played a role in his proof. Adjoint multiple zeta values are the quantities appearing as coefficients of the adjoint of the noncommutative generating series of multiple zeta values. They are known to generalize symmetric multiple zeta values.

In a series of his works, Jarossay defined an affine scheme $\AdDMR$ whose $R$-points are given by imposing the adjoint double shuffle relations for each commutative associative $\Q$-algebra $R$. He constructed an embedding
\[
\Ad(x_1):\DMR_0\hookrightarrow\AdDMR_0
\]
and asked whether this map is surjective or not \cite{Jarossay}.\footnote{It was pointed out in \cite{Anzawa-4} that this question requires a slight modification. Throughout this paper, $\AdDMR_0$ denotes the scheme denoted by $\AdDMR_0\cap\FAd$ in \cite{Anzawa-4}.}

In \cite{Anzawa-4}, we considered the tangent space $\widehat{\addmr}_0$ of $\AdDMR_0$ at its $\Q$-point $x_1$ and studied a necessary condition for a positive answer to the above question. We proved the following theorem.
\begin{thm}[{\cite[Main Theorem~3]{Anzawa-4}}]\label{thm:Anzawa-main}
The $\Q$-linear space $\widehat{\addmr}_0\cap\cFadred$ forms a Lie algebra together with some bracket.
\end{thm}
Here, $\cFadred$ is the $\Q$-linear space defined by the \emph{adjoint reduction relations} (see \eqref{eq:Fad}), which are $\Q$-linear relations derived from a functional identity satisfied by multitangent functions.\footnote{This space coincides with $\cFad\cap V_{\mathrm{strprty}}$ in \cite{Anzawa-4}.} We expect that
\[
\widehat{\addmr}_0\subset\cFadred,
\]
but have not yet proved this inclusion.

In this paper, we study the adjoint double shuffle relations together with the adjoint reduction relations. Our main theorem is the following.
\begin{mt}\label{mt:ism}
We have
\[
\AdDMR_0\cap\exp(\cFadred)\cong\DMR_0.
\]
\end{mt}

In this paper, we do not use the theory of schemes and just discuss the sets of $R$-points of the relevant schemes.

\begin{center}
\textbf{Acknowledgement}
\end{center}

I am grateful to almost everyone I met in Nagoya, from April 2020 to August 2026, for their help and support.

\begin{center}
\textbf{Use of AI}
\end{center}

The author used ChatGPT 6 Astra and 6.1 Sol to improve the writing and some proofs in this paper.
The main ideas of the proof in the depth-graded case were already obtained by the author around 2024. AI later suggested a simpler proof using the commutativity of $D$ and $\rho$, especially for \ref{eq:com-D-rho} and Lemma~\ref{lem:y1y1}. The author then used this simpler proof as a reference to construct the proof in the filtered case (Proposition~\ref{prop:ism}) given in this paper.

\section{Preliminaries}
We begin by preparing some notation and results which will be used throughout this paper.

\subsection*{The shuffle coproduct}

Let $X := \{x_0,x_1\}$ be the set of letters $x_0$ and $x_1$, $X^*$ the free monoid of words generated by $X$ with empty word $1$, and $\Q\langle X \rangle$ the non-commutative polynomial algebra over $\Q$ generated by $X$.
We define the $\Q$-algebra homomorphism $\Delta_{\shuffle} : \Q\langle X \rangle \rightarrow \Q\langle X \rangle^{\otimes 2}$, which is called \emph{shuffle coproduct}, by
\[
x_i \mapsto x_i \otimes 1 + 1\otimes x_i \qquad(i = 0,1).
\]
We also define the $\Q$-algebra anti-automorphism $S_X$ of $\Q\langle X \rangle$ by
\[
x_i \mapsto -x_i \qquad(i = 0,1).
\]
Then, together with the usual unit and counit maps, the tuple $(\Q\langle X \rangle, \textrm{conc}, \Delta_{\shuffle}, S_X)$ forms a Hopf algebra, where $\textrm{conc}$ denotes the concatenation product on $\Q\langle X \rangle$.

We define a degree, which we call the \emph{weight}, on $X^*$ by $\wt x_i = 1$ for $i = 0,1$, and endow $\Q\langle X \rangle$ with the weight topology,
i.e., 
for $n \in \mathbb{Z}_{\ge 0}$, we denote by $\Q\langle X \rangle_n$ the $\Q$-linear subspace of $\Q\langle X \rangle$ consisting of homogeneous elements of weight $n$. A fundamental system of neighborhoods of $0$ in $\Q\langle X \rangle$ is given by
\[
\bigoplus_{n\ge m} \Q\langle X \rangle_n
\]
for $m \in \mathbb{Z}_{\ge 0}$.
We note that the operations $\textrm{conc}$, $\Delta_{\shuffle}$, and $S_X$ are continuous with respect to the weight topology.
Thus, these operations extend continuously to the completion
\[
\Q\langle\langle X \rangle\rangle
:=
\varprojlim_{m}
\Q\langle X \rangle
\Big/
\bigoplus_{n\ge m}\Q\langle X \rangle_n.
\]
In particular, $\Delta_{\shuffle}$ extends to a continuous map
$\Q\langle\langle X \rangle\rangle \rightarrow \Q\langle\langle X \rangle\rangle^{\widehat{\otimes}2}$.
By abuse of notation, we use the same symbols $\Delta_{\shuffle}$ and $S_X$ for their continuous extensions.

We define $\langle \cdot \mid \cdot \rangle : \Q\langle \langle X \rangle\rangle \cotimes  \Q\langle X \rangle \rightarrow \Q$ given by
\[
\langle \Phi \mid w \rangle = \text{the coefficients of $\Phi$ on $w$ for $\Phi \in \Q\langle\langle  X \rangle\rangle$ and $w \in X^*$.}
\]
With this notation, every $\Phi \in \Q\langle\langle X \rangle\rangle$ can be written as
\[
\Phi = \sum_{w \in X^*} \langle \Phi \mid w \rangle w.
\]

For $\Q$-algebra $R$, we have
\[
R\langle X \rangle = R \otimes \Q\langle X \rangle, \qquad R \langle\langle X \rangle \rangle = R \cotimes \Q\langle\langle X \rangle \rangle.
\]
We use the above symbols or operations for these scalar extensions.

\subsection*{The completed free Lie algebra and its exponential}
In this paper, we mainly consider Lie algebras that are Lie subalgebras of the (completed) Lie algebra whose underlying space coincides with the (completed) free Lie algebra with $2$ generators.
In this subsection, we briefly recall some basic facts about free Lie algebras that will be used throughout this paper.

We define
\[
\F_2
:=
\{\psi \in \Q\langle X \rangle
\mid
\Delta_{\shuffle}(\psi)
=
\psi \otimes 1 + 1 \otimes \psi
\}.
\]
Then $\F_2$ is a Lie algebra with respect to the usual commutator bracket given by
\[
[f,g] := fg-gf
\qquad
(f,g\in \F_2).
\]
It is known that $\F_2$ coincides with the $\Q$-linear subspace of $\Q\langle X\rangle$ spanned by iterated commutators of $x_0$ and $x_1$ \cite[Theorem~1.4]{Reutenauer}.
Moreover, $\F_2$ is the free Lie algebra generated by $x_0$ and $x_1$ \cite[Theorem~0.5]{Reutenauer}.

We endow $\F_2$ with the relative topology of $\Q\langle X\rangle$, and denote by $\widehat{\F}_2$ its completion.
Then
\[
\widehat{\F}_2
=
\{\psi \in \Q\langle\langle X\rangle\rangle
\mid
\Delta_{\shuffle}(\psi)
=
\psi\otimes 1+1\otimes\psi
\}
\]
by \cite[Theorem~3.1]{Reutenauer} and $\widehat{\ide{F}}_2$ forms a pronilpotent Lie algebra.

In general, there is an one-to-one correspondence between pronilpotent Lie algebras and prounipotent groups.
Let
\[
\exp :
\prod_{m\geq 1}\Q\langle X\rangle_m
\longrightarrow
1+\prod_{m\geq 1}\Q\langle X\rangle_m;
\qquad
w\longmapsto
\sum_{n\geq 0}\frac{w^n}{n!}.
\]
We define
\[
\exp(\widehat{\F}_2)
:=
\{\exp(\psi)\mid \psi\in\widehat{\F}_2\}.
\]
Then, by \cite[Theorem~3.2]{Reutenauer},
\[
\exp(\widehat{\F}_2)
=
\{
\Phi\in\Q\langle\langle X\rangle\rangle
\mid
\langle \Phi \mid 1\rangle =1,\ 
\Delta_{\shuffle}(\Phi)=\Phi\widehat{\otimes}\Phi
\}.
\]
Moreover, $\exp(\widehat{\F}_2)$ forms a group under the concatenation product \cite[Corollary~3.3]{Reutenauer}.
Its corresponding pronilpotent Lie algebra is $\widehat{\F}_2$.

We introduce another Lie bracket on $\widehat{\F}_2$ and the corresponding group law on $\exp(\widehat{\F}_2)$ (cf. \cite[\S~3.1]{Racinet}). For $\Psi \in \Q\langle X\rangle$, let $d_{\Psi}$ be a derivation of $\Q\langle X\rangle$ defined by
\[
d_{\Psi}(x_0)=0,\qquad d_{\Psi}(x_1)=\Psi.
\]
For $\psi_1,\psi_2\in\Q\langle\langle X\rangle\rangle$, we define 
\[
\{\psi_1,\psi_2\} := s_{\psi_1}(\psi_2) - s_{\psi_2}(\psi_1).
\]
Here,
\[
s_{\psi_1}: \Q\langle X \rangle \rightarrow \Q\langle X \rangle \qquad ; \qquad \psi \mapsto d_{[x_1,\psi_1]}(\psi) + \psi_1\psi.
\]
Then $\F_2$ is closed under $\{\mathchar`-,\mathchar`-\}$, and \[ (\F_2,\{\mathchar`-,\mathchar`-\}) \] forms a Lie algebra.
We call $\{\mathchar`-,\mathchar`-\}$ the \emph{Ihara bracket}.
Since $s_{\psi}$ is continuous by the weight topology on $\ide{F}_2$, $s_{\psi}$ also extends to its completion $\widehat{\F}_2$.
We also denote the completed extension of $s_{\psi}$ and $\{\mathchar`-,\mathchar`-\}$ by same symbols.

For $\Phi\in \prod_{m\ge 1}\Q\langle X \rangle_m$, let $\kappa_\Phi$ be the continuous algebra endomorphism of $\Q\langle\langle X\rangle\rangle$ defined by 
\[
\kappa_\Phi(x_0)=x_0,\qquad \kappa_\Phi(x_1)=\Phi.
\]
For $\phi_1,\phi_2\in\exp(\widehat{\F}_2)$, we define 
\[
\phi_1\circledast \phi_2 := \phi_1\kappa_{\phi_1^{-1}x_1\phi}(\Phi_2). 
\] 
Then $\exp(\widehat{\F}_2)$ is closed under $\circledast$, and $ (\exp(\widehat{\F}_2),\circledast) $ forms a prounipotent group.
We call $\circledast$ the \emph{Ihara product}.
The Lie algebra associated with this group law is $ (\widehat{\F}_2,\{\mathchar`-,\mathchar`-\})$ and there is a bijection:
\begin{align*}
\exp^{\cast} : \widehat{\F_2} \rightarrow \exp(\F_2)\qquad;\qquad \phi \mapsto \sum_{n\ge 0} \frac{1}{n!} s_{\phi}^n(1). 
\end{align*}
Note that the exponential map $\exp^{\cast}$ is functorial with respect to scalar extensions of $\Q$-algebra.

These constructions extend to every commutative associative $\Q$-algebra $R$. Put
\[
\exp^{\cast}(\widehat{\F_2})(R) := \{\exp^{\circledast}(\psi) \mid \psi \in R\cotimes \widehat{\F_2}  \}.
\]
Then, $(R\cotimes \widehat{\F_2}, \{\mathchar`-,\mathchar`-\})$ forms a pronilpotent Lie algebra, $\exp(\widehat{\F_2}(R), \ast)$ forms a prounipotent group, and there is a bijection:
\[
\exp^{\cast} : R\cotimes \widehat{\F_2} \rightarrow \exp(\widehat{\F}_2)(R) \qquad;\qquad \phi \mapsto \sum_{n\ge 0} \frac{1}{n!} s_{\phi}^n(1). 
\]

\subsection*{The harmonic coproduct}

In this subsection, we review the theory of the harmonic coproduct.

Let $Y := \{y_n\}_{n \in \mathbb{Z}_{>0}}$ be the set of letters indexed by $\mathbb{Z}_{>0}$, $Y^*$ the free monoid of words generated by $Y$ with the empty word $1$, and $\Q\langle Y \rangle$ the non-commutative polynomial algebra over $\Q$ generated by $Y$.

We define the $\Q$-algebra homomorphism $\Delta_* : \Q\langle Y \rangle \rightarrow \Q\langle Y \rangle^{\otimes 2}$, called the \emph{harmonic coproduct}, by
\[
\Delta_*(y_k)
=
y_k\otimes 1
+
1\otimes y_k
+
\sum_{\substack{i+j=k\\i,j>0}} y_i\otimes y_j.
\]
Then, together with the usual unit and counit maps and the unique antipode, $(\Q\langle Y \rangle,\textrm{conc},\Delta_*)$ forms a Hopf algebra.

We define a degree, which we call the \emph{weight}, on $Y^*$ by $\wt y_k=k$ for $k\in\mathbb{Z}_{>0}$, and endow $\Q\langle Y\rangle$ with the weight topology. More precisely, for $n\in\mathbb{Z}_{\geq 0}$, we denote by $\Q\langle Y\rangle_n$ the $\Q$-linear subspace of $\Q\langle Y\rangle$ consisting of homogeneous elements of weight $n$. A fundamental system of neighborhoods of $0$ in $\Q\langle Y\rangle$ is given by
\[
\bigoplus_{n\geq m}\Q\langle Y\rangle_n
\]
for $m\in\mathbb{Z}_{\geq 0}$.

We note that the concatenation product and $\Delta_*$ are continuous with respect to the weight topology. Thus, these operations extend continuously to the completion
\[
\Q\langle\langle Y\rangle\rangle
:=
\varprojlim_m
\Q\langle Y\rangle
\Big/
\bigoplus_{n\geq m}\Q\langle Y\rangle_n.
\]
In particular, $\Delta_*$ extends to a continuous map
$\Q\langle\langle Y\rangle\rangle
\rightarrow
\Q\langle\langle Y\rangle\rangle^{\widehat{\otimes}2}$.
By abuse of notation, we use the same symbol $\Delta_*$ for its continuous extension.

We define the pairing
\[
\langle \,\cdot\, \mid \,\cdot\, \rangle
:
\Q\langle\langle Y\rangle\rangle
\times
\Q\langle Y\rangle
\longrightarrow
\Q
\]
by letting $\langle \Phi\mid w\rangle$ denote the coefficient of $w$ in $\Phi$ for $\Phi\in\Q\langle\langle Y\rangle\rangle$ and $w\in Y^*$. With this notation, $\Phi\in\Q\langle\langle Y\rangle\rangle$ can be written as
\[
\Phi
=
\sum_{w\in Y^*}
\langle \Phi\mid w\rangle w.
\]

We define the $\Q$-linear map
\[
\pi_Y:\Q\langle X\rangle\longrightarrow\Q\langle Y\rangle
\]
by
\[
\pi_Y
\left(
x_0^{k_1-1}x_1
\cdots
x_0^{k_r-1}x_1x_0^l
\right)
=
\begin{cases}
y_{k_1}\cdots y_{k_r} & \text{if } l=0,\\
0 & \text{if } l>0,
\end{cases}
\]
for $r\geq 1$, $k_1,\ldots,k_r\in\mathbb{Z}_{>0}$, and $l\in\mathbb{Z}_{\geq 0}$, together with $\pi_Y(1)=1$.
Since $\pi_Y$ is continuous with respect to the weight topologies on $\Q\langle X\rangle$ and $\Q\langle Y\rangle$, it extends uniquely to a continuous map
\[
\pi_Y:
\Q\langle\langle X\rangle\rangle
\longrightarrow
\Q\langle\langle Y\rangle\rangle.
\]
By abuse of notation, we denote this continuous extension again by $\pi_Y$.

To conclude this subsection, for $\Q$-algebra $R$, we have
\[
R\langle Y \rangle = R \otimes \Q\langle Y \rangle, \qquad R \langle\langle Y \rangle \rangle = R \cotimes \Q\langle\langle Y \rangle \rangle.
\]
We use the above symbols or operations for these scalar extensions.

\section{Double shuffle Lie algebra and group}
In this section, we briefly review the double shuffle Lie algebra and the double shuffle group introduced by Racinet \cite{Racinet}.
\subsection*{The double shuffle Lie algebra}
We first define some Lie algebra following \cite[\S\S~3.3.1 and~3.3.8]{Racinet}.
\begin{df}\label{df:ds}
We define the \emph{double shuffle Lie algebra} $\dmr_0$ (see next theorem) to be the $\Q$-linear subspace of $\F_2$ consisting of elements $\phi$ satisfying the following conditions:
\begin{enumerate}
\item\label{df:ds-1} $\langle \phi \mid 1 \rangle = 0$.
\item\label{df:ds-2} $\langle \phi \mid x_0 \rangle = \langle \phi \mid x_1 \rangle = 0$.
\item\label{df:ds-3} $\langle \phi \mid x_0x_1 \rangle = 0$.
\item\label{df:ds-4} $\Delta_*(\phi_*) = \phi_* \otimes 1 + 1\otimes \phi_*$, where
\begin{align*}
\phi_*
:=&
\pi_Y(\phi)
+
\phi_{\corr}\\
\phi_{\corr}=&
\sum_{n\geq 2}
\frac{(-1)^{n-1}}{n}
\langle \phi \mid x_0^{n-1}x_1 \rangle y_1^n.
\end{align*}
\end{enumerate}
\end{df}
Racinet proved the following result.
\begin{thm}[{Corollary of \cite[Proposition~4.A.i]{Racinet}}]\label{thm:ds-Lie}
The pair
\[
(\dmr_0,\{\mathchar`-,\mathchar`-\})
\]
forms a Lie algebra.
\end{thm}
More precisely, Racinet showed the above theorem in the case of the completed case of $\dmr_0$.
We endow $\dmr_0$ with the topology induced by the weight topology on $\Q\langle X\rangle$, and denote its completion by $\widehat{\dmr_0}$. Then
\[
(\widehat{\dmr_0},\{\mathchar`-,\mathchar`-\})
\]
is a pronilpotent Lie algebra.
\begin{rmk}
The completed double shuffle Lie algebra $\widehat{\dmr_0}$ can equivalently be described as the $\Q$-linear subspace of $\widehat{\F}_2$ consisting of elements $\phi$ satisfying the following conditions:
\begin{enumerate}
\item $\langle \phi \mid 1 \rangle = 0$.
\item $\langle \phi \mid x_0 \rangle = \langle \phi \mid x_1 \rangle = 0$.
\item $\langle \phi \mid x_0x_1 \rangle = 0$.
\item$\Delta_*(\phi_*) = \phi_* \otimes 1 + 1\otimes \phi_*$.
\end{enumerate}
\end{rmk}

\subsection*{The double shuffle group}
In this subsection, we review the double shuffle group.
Throughout this subsection, let $R$ be a commutative $\Q$-algebra.
For $\lambda\in R$, we define the following set. These assignments are functorial in $R$.
\begin{df}\label{df:ds-group}
We define the \emph{double shuffle torsor} $\DMR_{\lambda}(R)$ (see next theorem) to be the subset of $\exp\bigl(\widehat{\F}_2\bigr)(R)$
consisting of elements $\Phi$ satisfying the following conditions:
\begin{enumerate}
\item $\langle \Phi \mid x_0 \rangle = \langle \Phi \mid x_1 \rangle = 0$.
\item
\[
\langle \Phi \mid x_0x_1 \rangle
=
-\frac{\lambda^2}{24}.
\]
\item
\[
\Delta_*(\Phi_*)
=
\Phi_*\otimes\Phi_*,
\]
where
\[
\Phi_*
:=
\exp\left(
\sum_{n\geq 2}
\frac{(-1)^{n-1}}{n}
\langle \Phi \mid x_0^{n-1}x_1\rangle y_1^n
\right)
\pi_Y(\Phi).
\]
\end{enumerate}
We define the \emph{double shuffle group} $\DMR_0(R)$ (also see next theorem) by
\[
\DMR_0(R):=\DMR_{\lambda}(R)\big|_{\lambda=0}.
\]
Equivalently, $\DMR_0(R)$ consists of the elements $\Phi$ satisfying the conditions above together with
\[
\langle \Phi \mid x_0x_1\rangle=0.
\]

\end{df}

Racinet showed the following fundamental result.

\begin{thm}[{\cite[Theorem~I]{Racinet}}]\label{thm:ds-group}

For every commutative $\Q$-algebra $R$, the following assertions hold:

\begin{enumerate}
\item The set $\DMR_0(R)$ is a prounipotent subgroup of
$
(\exp\bigl(\widehat{\F}_2\bigr)(R),
\circledast).$

\item For each $\lambda\in R$, the natural left action of $\DMR_0(R)$ on $\DMR_{\lambda}(R)$ induced by the Ihara product $\circledast$ is free and transitive. In other words, $\DMR_{\lambda}(R)$ is a $\DMR_0(R)$-torsor.

\item The Lie algebra associated with the group of $\DMR_0(R)$ is $R\cotimes\widehat{\dmr_0}.$

\end{enumerate}

\end{thm}

\section{Adjoint double shuffle space and adjoint double shuffle set}

\subsection*{Preliminary}

We first prepare some tools. Let $T$ be an indeterminate,
\[
\Q[T;Y\rangle := \Q[T]\langle Y \rangle,
\]
and
\[
\pi_{Y,\#}\colon \Q\langle  X \rangle \longrightarrow\Q[T;Y\rangle
\]
the $\Q$-linear map defined by
\[
\pi_{Y,\#}\bigl(x_0^{k_1-1}x_1 x_0^{k_2-1}x_1 \cdots x_0^{k_r-1}x_1 x_0^{l}\bigr)
:=
\begin{cases}
T^{k_1-1} y_{k_2}\cdots y_{k_r} & \text{if } l=0,\\
0 & \text{otherwise.}
\end{cases}
\]
Recall that $\Q\langle Y\rangle$ is endowed with the weight filtration, and $\Q\langle Y\rangle_{\ge m}\subset \Q\langle Y\rangle$ is the $\Q$-subspace spanned by words of weight $\ge m$.
We endow $\Q[T;Y\rangle$ with the grading determined by $\wt T=1$ and $\wt y_m=m$ for $m\in \mathbb{Z}_{\ge 1}$.
For $d\ge0$, let $\Q[T;Y\rangle_{d}$ denote the homogeneous component of total weight $d$.
For $m\ge1$, set
\[
A_m :=\Q[T;Y\rangle \Big/ \bigoplus_{d\ge m} \Q[T;Y\rangle_{d} .
\]
Then the natural truncation maps $A_{m'}\rightarrow A_m$ $(m\le m')$ form a projective system, and we define
\begin{equation*}\label{eq:lim-com}
\Q[[T;Y \rangle\rangle
:= \varprojlim_{m\ge1} A_m .
\end{equation*}
Note that  the canonical inclusion $\Q[T] \hookrightarrow \Q[T; Y \rangle$ induces a canonical action $\Q[T]/(T)^m \rightarrow  A_m$ for each $m \in \mathbb{Z}_{\ge 1}$.
Passing to the projective limit, $\Q[[T; Y \rangle\rangle$ is a $\Q[[T]]$-algebra.

In the same manner, we define $R$-algebra $R[[T;Y \rangle\rangle$ as
\[
R[[T;Y \rangle\rangle
\cong \varprojlim_{m} R\otimes A_m
\]
and we can see $R[[T;Y \rangle\rangle$ as $R[[T]]$-algebra.

By definition, for each $n\ge 1$ we have
\[
\pi_{Y,\#}(X^{n})
\subset
\Q[T;Y\rangle_{n-1}
\subset\ \Q[T; Y\rangle .
\]
Thus, $\pi_{Y,\#}$ is a continuous $\Q$-linear map on $\Q\langle X \rangle$ endowed with the weight topology,
and hence it extends uniquely to a continuous $\Q$-linear map $\Q\langle \langle X \rangle\rangle \longrightarrow \Q[[T;Y \rangle\rangle .$ 
We use the same symbol for the continuous extension and the scalar extensions.

\subsection*{Adjoint aspects of $\ide{F}_2$ and $\exp(\widehat{\F_2})$}

Put
\[
\Fad
:=
\{\Psi \in \ide{F}_2
\mid
{}^{\exists}\psi\in\ide{F}_2
\text{ such that }
\Psi=[x_1,\psi]\}.
\]
We define the $\Q$-bilinear operator $\{\mathchar`-,\mathchar`-\}_1$ on $\Q\langle X\rangle$ by
\[
\{\Psi_1,\Psi_2\}_1
:=
d_{\Psi_1}(\Psi_2)-d_{\Psi_2}(\Psi_1)
\]
for $\Psi_1,\Psi_2\in\Q\langle X\rangle$. Then, the following holds.

\begin{prop}[{cf. \cite[Proposition~5.1]{Anzawa-4}}]

The pair
\[
(\Fad,\{\mathchar`-,\mathchar`-\}_1)
\]
forms a Lie algebra.

\end{prop}

We endow $\Fad$ with the topology induced by the weight topology and denote its completion by $\cFad$.
The map
\[
\ad(x_1):
\ide{F}_2
\longrightarrow
\Fad;
\qquad
\psi\longmapsto[x_1,\psi]
\]
is a continuous surjective map with kernel $\Q x_1$.
Moreover, it satisfies
\[
\ad(x_1)(\{\psi_1,\psi_2\})
=
\{\ad(x_1)(\psi_1),\ad(x_1)(\psi_2)\}_1
\]
for $\psi_1,\psi_2\in\ide{F}_2$.

The group aspect of $\cFad$ is described as follows. Put
\[
\FAd
:=
\{
\phi^{-1}x_1\phi
\mid
\phi\in\exp(\widehat{\F}_2)
\}.
\]
We define a binary operation $\cast_1$ on $\FAd$ by
\[
\Phi_1\cast_1\Phi_2
:=
\kappa_{\Phi_1}(\Phi_2)
\qquad
(\Phi_1,\Phi_2\in\FAd).
\]
Then, the following holds.

\begin{prop}[{cf. \cite[Remark~4.13]{Anzawa-4}}]

The pair
\[
(\FAd,\cast_1)
\]
forms a group with unit $x_1$.

\end{prop}

Moreover, the map
\[
\Ad(x_1):
\exp(\ide{F}_2)
\longrightarrow
\FAd;
\qquad
\phi\longmapsto\phi^{-1}x_1\phi
\]
is a continuous surjective group homomorphism, where $\exp(\ide{F}_2)$ is equipped with the Ihara product and $\FAd$ with $\cast_1$.
Its kernel is $\exp(\Q x_1)$.
The Lie algebra corresponding to $(\FAd,\cast_1)$ is
\[
(\cFad,\{\mathchar`-,\mathchar`-\}_1).
\]
One should note that an exponential map between $\cFad$ and $\FAd$ is provided by
\begin{align}
\exp^{\cast_1} : \cFad \rightarrow \FAd(\Q) \quad ; \quad \Psi \mapsto \sum_{n\ge 0} \frac{d_{\Psi}^n}{n!}(x_1). \label{eq:exp-ad}
\end{align}
Note that there is a commutative diagram:
\begin{align}
\xymatrix{
\exp(\widehat{\F}_2) \ar[r]^{\Ad(x_1)} &\FAd \\
\widehat{\F}_2 \ar[u]^{\exp^{\cast}} \ar[r]_{\ad(x_1)} & \cFad \ar[u]_{\exp^{\cast_1}}
} \label{eq:dia-exp}
\end{align}

To finish this subsection,  
put
\[
\FAd(R)
:=
\left\{
\phi^{-1}x_1\phi
\mid
\phi\in\exp\left(R\widehat{\otimes}\widehat{\F}_2\right)
\right\}.
\]
Note that \eqref{eq:dia-exp} is functorial with respect to $\Q$-algebras.

\subsection*{The adjoint double shuffle linear space}

We next introduce the linearized version of the adjoint double shuffle relations.

For $\Psi\in\Q\langle\langle X\rangle\rangle$, put
\[
\Psi_{\#}:=\pi_{Y,\#}(\Psi)\in\Q[[T;Y\rangle\rangle.
\]
Here, we extend $\Delta_*$ continuously to $\Q[[T;Y\rangle\rangle$ as a $\Q[[T]]$-algebra homomorphism.

\begin{df}[{cf. \cite[Definition~3.2.6]{Jarossay}}]\label{df:addmr}
We define the \emph{adjoint double shuffle linear space} $\addmr_0$ to be the $\Q$-linear subspace of $\Fad$ consisting of elements $\Psi$ satisfying the following conditions:
\begin{enumerate}
\item
\[
\langle\Psi\mid w\rangle=0
\]
for every $w\in X^*$ with $\wt w<4$.
\item
\[
\Delta_*(\Psi_{\#})
=
\Psi_{\#}\otimes 1+1\otimes\Psi_{\#}.
\]
\end{enumerate}
\end{df}

We endow $\addmr_0$ with the topology induced by the weight topology on $\Fad$, and denote its completion by $\widehat{\addmr_0}$. Then $\widehat{\addmr_0}$ is naturally regarded as a $\Q$-linear subspace of $\cFad$.
Equivalently, $\widehat{\addmr}_0$ is the tangent space at $x_1$ of the adjoint double shuffle set defined in the next subsection.

\subsection*{The adjoint double shuffle set}

Throughout this subsection, let $R$ be a commutative $\Q$-algebra.
We extend $\pi_{Y,\#}$ and $\Delta_*$ to $R$ by scalar extension.
For $\Phi\in R\langle\langle X\rangle\rangle$, put
\[
\Phi_{\#}:=\pi_{Y,\#}(\Phi)\in R[[T;Y\rangle\rangle.
\]
\begin{df}[{\cite[Definition~3.2.6]{Jarossay}}]\label{df:AdDMR}

We define the \emph{adjoint double shuffle set} $\AdDMR_0(R)$ to be the subset of $\FAd(R)$ consisting of elements $\Phi$ satisfying the following conditions:

\begin{enumerate}

\item
\[
\langle\Phi-x_1\mid w\rangle=0
\]
for every $w\in X^*$ with $\wt w<4$.

\item
\[
\Delta_*(\Phi_{\#})
=
\Phi_{\#}\otimes\Phi_{\#}.
\]

\end{enumerate}

\end{df}
Moreover, for every commutative $\Q$-algebra $R$, the map
\[
\Ad_R(x_1):
\DMR_0(R)
\longrightarrow
\AdDMR_0(R);
\qquad
\phi\longmapsto\phi^{-1}x_1\phi
\]
is well-defined and injective (cf. \cite[Proposition~3.2.5]{Jarossay}).

In particular, by taking $R=\Q[\varepsilon]/(\varepsilon^2)$, one obtains
\[
\widehat{\addmr}_0
=
\left\{
\Psi\in\cFad
\mid
x_1+\varepsilon\Psi
\in
\AdDMR_0\bigl(\Q[\varepsilon]/(\varepsilon^2)\bigr)
\right\}.
\]

\begin{rmk}

The notation in this paper differs slightly from that in \cite{Anzawa-4}.
The space denoted by $\addmr_0$ in the present paper is the space denoted by
$\addmr_0\cap\Fad$ in \cite{Anzawa-4}.
Similarly, the set denoted by $\AdDMR_0(R)$ in the present paper is the set
$\AdDMR_0(R)\cap\FAd(R)$ in the notation of \cite{Anzawa-4}.
In other words, throughout this paper, the adjoint condition is incorporated into the definitions of $\addmr_0$ and $\AdDMR_0(R)$.

\end{rmk}

\section{The refined versions of the Jarossay's questions, results, and main  theorems}

\subsection*{The refined Jarossay's questions}

In \cite{Anzawa-4}, refined versions of Jarossay's question (\cite[Question~3.2.8]{Jarossay}) were formulated on both the group side and the corresponding Lie algebra side. In the notation of the present paper, we consider the following questions.

\begin{qu}[{cf. \cite[Questions~4.10 and~4.12]{Anzawa-4}}]\label{qu:Jarossay}
\begin{enumerate}
\item Does $\AdDMR_0(R)$ form a subgroup of $(\FAd(R),\cast_1)$, and is the adjoint map
\[
\Ad_R(x_1):
\DMR_0(R)
\longrightarrow
\AdDMR_0(R)
\]
a group isomorphism for every commutative $\Q$-algebra $R$?
\item Does $\widehat{\addmr_0}$ form a Lie subalgebra of
$(\cFad,\{\mathchar`-,\mathchar`-\}_1)$, and is the adjoint map
\[
\ad(x_1):
\widehat{\dmr_0}
\longrightarrow
\widehat{\addmr_0}
\]
a Lie algebra isomorphism?
\item Does $\addmr_0$ form a Lie subalgebra of
$(\Fad,\{\mathchar`-,\mathchar`-\}_1)$, and is the adjoint map
\[
\ad(x_1):
\dmr_0
\longrightarrow
\addmr_0
\]
a Lie algebra isomorphism?
\end{enumerate}
\end{qu}
One motivation for these questions comes from the comparison of the weight-homogeneous components of $\dmr_0$ and $\addmr_0$. As observed in \cite[Table~1]{Anzawa-4}, their corresponding homogeneous components have the same dimensions up to weight $10$, with the weight shifted by one under $\ad(x_1)$.

\subsection*{A result related to Question~\ref{qu:Jarossay}}

To describe the result related to Question~\ref{qu:Jarossay}, we introduce another $\Q$-linear subspace of $\Fad$.
For $\Psi\in\Q\langle X\rangle$ with no terms of weight less than $2$, we write uniquely
\[
\Psi
=
\sum_{i,j\in\{0,1\}}
x_i\Psi^{ij}x_j,
\qquad
\Psi^{ij}\in\Q\langle X\rangle.
\]
We define
\begin{align}
\Fadred
:=
\left\{
\Psi = [x_1,\psi]\in\Fad
\ \middle|\
\begin{array}{l}
\psi \text{ has no linear term},\\
\Psi^{11}+\Psi^{10}+\Psi^{01}=0
\end{array}
\right\}. 
\end{align}
Then $\Fadred$ is a Lie subalgebra of
$(\Fad,\{\mathchar`-,\mathchar`-\}_1)$.
We endow $\Fadred$ with the topology induced by the weight topology on $\Fad$, and denote its completion by $\cFadred$. Thus,
\[
\cFadred\subset\cFad
\]
is a closed Lie subalgebra.

\begin{rmk}
In \cite{Hirose-parity}, Hirose derived the following explicit formula for the parity relations of multiple zeta values using an identity for multitangent functions due to Bouillot \cite{Bouillot}:
\begin{align}
\begin{aligned}
&\sum_{j = 0}^r (-1)^{k_{j+1}+\cdots + k_r} \zeta^*(k_{j},\ldots,k_1)\zeta^*(k_{j+1},\ldots,k_r) \\
=&\delta^{k_1,\ldots,k_r} + \sum_{j = 1}^r \sum_{\substack{a+2m+b = k_j \\ a,b,m\ge 0}} (-1)^{b+k_{j+1}+\cdots +k_r + m + 1}\\
&\times \frac{(2\pi )^{2m}}{(2m)!}B_{2m}\zeta_a^*(k_{j-1},\ldots,k_1)\zeta_b^*(k_{j+1},\ldots,k_r).
\end{aligned}
\label{eq:parity}
\end{align}
Here, $(k_1,\ldots,k_r)\in\mathbb{Z}_{>0}^r$,
$\zeta^*(k_1,\ldots,k_r)$ denotes the harmonically regularized multiple zeta value (see \cite[Proposition~1]{Ihara-Kaneko-Zagier}),
and $\delta^{k_1,\ldots,k_r}$ is an element of $\Q[\pi^2]$.
We follow the notation of \cite{Hirose-parity} for $\zeta_a^*$.

In \cite{Kawamura-Maesaka}, Kawamura and Maesaka refer to \eqref{eq:parity} as the \emph{reduction relation} and show that the affine scheme defined by the corresponding formal relations can be identified with $\mathcal{G}^{\mathrm{inert}}$ introduced by Enriquez and Furusho \cite{Enriquez-Furusho}.
Enriquez and Furusho showed that $\mathcal{G}^{\mathrm{inert}}$ is isomorphic to a certain subgroup of the group of tangential automorphisms of $\exp(\ide{F}_2)$ that preserve the conjugacy classes of $\exp(x_0)$, $\exp(x_1)$, and $\exp(x_\infty)$, where $x_\infty=-x_0-x_1$.
This is closely related to the speciality condition in the Kashiwara--Vergne theory developed in \cite{Alekseev-Torossian}.

In \cite{Anzawa-4}, we studied the algebraic relations
\begin{align}
\Psi^{11}+\Psi^{10}+\Psi^{01}=0 \label{eq:Fad}
\end{align}
in connection with \eqref{eq:parity} and referred to them as the strong parity relations.
In view of \cite{Kawamura-Maesaka}, we now refer to \eqref{eq:Fad} as the \emph{adjoint reduction relations}.
\end{rmk}

We can now restate the main theorem of \cite{Anzawa-4} in the notation of the present paper.
\begin{thm}[{\cite[Main Theorem~3]{Anzawa-4}}]\label{thm:Anzawa-main}
The $\Q$-linear space $\widehat{\addmr_0}\cap\cFadred$
is a Lie subalgebra of $(\cFad,\{\mathchar`-,\mathchar`-\}_1).$
\end{thm}
By taking the direct sum of the weight-homogeneous components, we immediately obtain the non completed part.
\begin{cor}\label{cor:Anzawa-main-finite}
The $\Q$-linear space $\addmr_0\cap\Fadred $ is a Lie subalgebra of $(\Fad,\{\mathchar`-,\mathchar`-\}_1).$
\end{cor}

\subsection*{Main Theorems}

We now state the following main theorem of this paper.
\begin{thm}[ = Main Theorem~\ref{mt:ism}]\label{thm:ism}
For a commutative associative $\Q$-algebra $R$, put
\[
\exp(\cFadred)(R): = \exp^{\cast_1}(R\cotimes \cFadred).
\]
Then, up to scalar extensions of $R$, the map $\Ad(x_1)$ induces the following group isomorphism
\[
\DMR_0(R) \xrightarrow{\sim} \AdDMR_0(R) \cap \exp(\cFadred)(R).
\]
Moreover, $\exp^{\cast_1}$ gives the functorial bijection
\[
R\cotimes (\widehat{\addmr}_0\cap\cFadred) \xrightarrow{\sim} \AdDMR_0(R) \cap \exp(\cFadred)(R)
\]
with respect to $R$.
\end{thm}

To prove Theorem~\ref{thm:ism}, we have to show
\begin{itemize}
\item $\ad(x_1)(\dmr_0) \subset \Fadred$.
\item The surjectivity of $\ad(x_1)|_{\dmr_0}$.
\end{itemize}
Injectivity is immediate because $\ker \ad(x_1) = \Q x_1$ and all elements of $\dmr_0$ have no linear terms,
We prove Theorem~\ref{thm:ism} in the following section.

\section{A proof of Theorem~\ref{thm:ism} (=Main Theorem~\ref{mt:ism})}

An essential proposition to prove Theorem~\ref{thm:ism} is the following:

\begin{prop}\label{prop:ism}
The adjoint map $\ad(x_1): \dmr_0 \longrightarrow \addmr_0\cap\Fadred; \phi\longmapsto [x_1,\phi]$ is a Lie algebra isomorphism.
\end{prop}
Though the fact $\ad(x_1)(\dmr_0) \subset \addmr_0$ is already proved,  $\ad(x_1)(\dmr_0) \subset \Fadred$ is not.
In our proof of Proposition~\ref{prop:ism}, we first prove that, under $\ad(x_1)(\dmr_0) \subset \Fadred$, $\psi \in \dmr_0$ holds if $[x_1,\psi] \in \addmr_0 \cap \Fadred$.
We prove such inclusion later.

\subsection*{Some operations}

We first define $D$ as a derivation on $\Q\langle Y \rangle$ such that
\[
D(y_1)=0,\qquad D(y_{n+1})=ny_n
\]
for $n\geq 1$, and define $\rho$ as the cyclic operator on
$\Q\langle Y\rangle$ such that
\[
\rho(1)=1,\qquad
\rho(y_{i_1}\cdots y_{i_s})
=
y_{i_s}y_{i_1}\cdots y_{i_{s-1}},
\qquad
(i_1,\ldots,i_s\in\mathbb Z_{>0}).
\]

Then, the following holds.

\begin{lem}{[\cite[(1) = \S~4.2.2, (3) Remark below Corollary~4.2.6]{Racinet}}]\label{lem:D-rho}
\begin{enumerate}
\item\label{lem:D-rho-1}
For $\phi\in\ide{F}_2$ with no linear term, we have
\begin{align}
D(\pi_Y(\phi x_1))=0.
\label{eq:D-vanish}
\end{align}

\item\label{lem:D-rho-2}
We have
\begin{align}
D\circ\rho=\rho\circ D.
\label{eq:com-D-rho}
\end{align}

\item\label{lem:D-rho-3}
The map $D$ is a coderivation with respect to $\Delta_*$, namely,
\begin{align}
\Delta_*\circ D
=
(D\otimes\id+\id\otimes D)\circ\Delta_*.
\label{eq:D-coderivation}
\end{align}
\end{enumerate}
\end{lem}

\begin{proof}
(\ref{lem:D-rho-1}).
Let $\del_{x_0}$ be the derivation on $\Q\langle X\rangle$ given by
\[
\del_{x_0}(x_0)=1,\qquad
\del_{x_0}(x_1)=0.
\]
As indicated in \cite[\S\S 4.2]{Racinet}, we have
\[
\del_{x_0}(\phi)=0.
\]
Hence,
\[
\del_{x_0}(\phi x_1)
=
\del_{x_0}(\phi)x_1
+
\phi\del_{x_0}(x_1)
=
0.
\]
Since
\[
\pi_Y\circ\del_{x_0}\mid_{\F_2x_1}
=
D\circ\pi_Y\mid_{\F_2x_1}  ,
\]
we obtain
\[
D(\pi_Y(\phi x_1))
=
\pi_Y(\del_{x_0}(\phi x_1))
=
0.
\]

(\ref{lem:D-rho-2}).
For $y_{k_1}\cdots y_{k_r} \in Y^*$, we have
\begin{align*}
\rho\circ D(y_{k_1}\cdots y_{k_r})
&=
\sum_{\substack{1\leq i\leq r\\ k_i\geq2}}
(k_i-1)
\rho(
y_{k_1}\cdots y_{k_{i-1}}
y_{k_i-1}
y_{k_{i+1}}\cdots y_{k_r})\\
&=
D(y_{k_r}y_{k_1}\cdots y_{k_{r-1}})\\
&=
D\circ\rho(y_{k_1}\cdots y_{k_r}),
\end{align*}
as required.

(\ref{lem:D-rho-3}).
Put $y_0:=1$. Recall that
\[
\Delta_*(y_n)
=
\sum_{i+j=n}y_i\otimes y_j.
\]
Then we have
\begin{align*}
&(D\otimes\id+\id\otimes D)\Delta_*(y_{1})\\
=&(D\otimes\id+\id\otimes D)(y_1 \otimes 1 + 1\otimes y_1)\\
=&0\\
=&\Delta_*(D(y_1))
\end{align*}
and, for $n\geq1$, we have
\begin{align*}
&(D\otimes\id+\id\otimes D)\Delta_*(y_{n+1})\\
&=
(D\otimes\id+\id\otimes D)
\left(
\sum_{i+j=n+1}y_i\otimes y_j
\right)\\
&=
n\sum_{i+j=n}y_i\otimes y_j\\
&=
\Delta_*(ny_n)\\
&=
\Delta_*\circ D(y_{n+1}).
\end{align*}
Since $\Delta_*$ is an algebra homomorphism and $D$ is a derivation,
the assertion follows.
\end{proof}

\subsection*{Auxiliary series}

For $\Psi\in\Fad$, we take $\psi\in\F_2$ with $\Psi=[x_1,\psi]$ and $\langle\psi\mid x_1\rangle=0$.
Throughout this section, we assume that $\langle\psi\mid x_0\rangle=0$, which is equivalent to $\langle\Psi\mid x_1x_0\rangle=0$.
Put
\[
\pi_Y(\psi x_1)
=
\sum_{n\geq0}y_{n+1}\Psi^{(n)}
=
\sum_{n\geq0}\Psi_{(n)}y_{n+1}
\]
and define $R_n$ by
\[
\pi_Y(\psi x_1)
-
\rho\circ\pi_Y(\psi x_1)
=
\sum_{n\geq0}y_{n+1}R_n.
\]
Since $\rho\bigl(\sum_{n\ge0}\Psi_{(n)}y_{n+1}\bigr)=\sum_{n\ge0}y_{n+1}\Psi_{(n)}$, we have
\[
R_n=\Psi^{(n)}-\Psi_{(n)}\qquad(n\ge0).
\]
Moreover, since $\pi_Y$ annihilates the words ending with $x_0$ and $\langle\psi\mid1\rangle=0$, the part of $\pi_Y(\psi x_1)$ ending with $y_1$ is $\pi_Y(\psi)y_1$, that is,
\[
\Psi_{(0)}=\pi_Y(\psi).
\]

\begin{lem}\label{lem:recurrent}
For every $n\geq0$, we have
\begin{align}
(n+1)R_{n+1}+D(R_n)=0.
\label{eq:reccurent}
\end{align}
\end{lem}

\begin{proof}
By Lemma~\ref{lem:D-rho}, we have
\begin{align*}
&D\left(
\pi_Y(\psi x_1)
-
\rho\circ\pi_Y(\psi x_1)
\right)\\
\stackrel{\eqref{eq:com-D-rho}}{=} &
D(\pi_Y(\psi x_1))
-
\rho\circ D(\pi_Y(\psi x_1))\\
\stackrel{\eqref{eq:D-vanish}}{=}&
0.
\end{align*}
Therefore,
\begin{align*}
0
&=
D\left(
\sum_{n\geq0}y_{n+1}R_n
\right)\\
&=
\sum_{n\geq0}
\left(
ny_nR_n+y_{n+1}D(R_n)
\right)\\
&=
\sum_{n\geq0}
y_{n+1}
\left(
(n+1)R_{n+1}+D(R_n)
\right).
\end{align*}
Hence,
\[
(n+1)R_{n+1}+D(R_n)=0
\]
as asserted.
\end{proof}

We first determine $R_0$.

\begin{lem}\label{lem:K0}
For $\Psi=[x_1,\psi]\in\Fad$, we have
\[
R_0=\Psi^{(0)}-\pi_Y(\psi).
\]
\end{lem}

\begin{proof}
By definition, the part of $\pi_Y(\psi x_1)$ beginning with $y_1$ is $y_1\Psi^{(0)}.$
On the other hand, the part of
$\rho\circ\pi_Y(\psi x_1)$ beginning with $y_1$ is $y_1\pi_Y(\psi).$
Hence, $R_0=\Psi^{(0)}-\pi_Y(\psi)$, as claimed.
\end{proof}

\begin{lem}\label{lem:Phi-com}
For $\Psi=[x_1,\psi]\in\Fad$, we have
\begin{align}
\sum_{n\geq0}[y_{n+1},\Psi^{(n)}]
=
- \sum_{n\ge 0} R_{n}y_{n+1}
\label{eq:Phi-com}
\end{align}
\end{lem}

\begin{proof}
Since
\[
\pi_Y(\psi x_1)
=
\sum_{n\geq0}y_{n+1}\Psi^{(n)},
\]
we have
\[
\rho^{-1}\circ\pi_Y(\psi x_1)
=
\sum_{n\geq0}\Psi^{(n)}y_{n+1}.
\]
Therefore,
\begin{align*}
&\rho^{-1}\left(
\pi_Y(\psi x_1)
-
\rho\circ\pi_Y(\psi x_1)
\right)\\
&\qquad=
\sum_{n\geq0}
\left(
\Psi^{(n)}y_{n+1}
-
y_{n+1}\Psi^{(n)}
\right)\\
&\qquad=
-\sum_{n\geq0}[y_{n+1},\Psi^{(n)}].
\end{align*}
On the other hand, by the definition of $R_n$,
\[
\rho^{-1}
\left(
\sum_{n\geq0}y_{n+1}R_n
\right)
=
\sum_{n\geq0}R_ny_{n+1}.
\]
Hence,
\[
\sum_{n\geq0}[y_{n+1},\Psi^{(n)}]
=
-\sum_{n\geq0}R_ny_{n+1},
\]
as required.
\end{proof}

We now retake
\[
\Psi=[x_1,\psi]\in\addmr_0.
\]
Then, by comparison with the coefficients of powers of $T$,
\[
\Delta_*(\pi_{Y,\#}(\Psi))
=
\pi_{Y,\#}(\Psi)\otimes1
+
1\otimes\pi_{Y,\#}(\Psi)
\]
is equivalent to
\begin{align*}
&\Delta_*(\Psi^{(n)})
=
\Psi^{(n)}\otimes1
+
1\otimes\Psi^{(n)},
& n\ge 1,\\
&\Delta_*(\pi_Y(\psi)-\Psi^{(0)})
=
(\pi_Y(\psi)-\Psi^{(0)})\otimes1
+
1\otimes(\pi_Y(\psi)-\Psi^{(0)}).&
\end{align*}

By Lemma~\ref{lem:K0},
\[
R_0
=
-\left(
\pi_Y(\psi)-\Psi^{(0)}
\right),
\]
and hence $\Delta_*(R_0) = R_0 \otimes 1 + 1\otimes R_0$.

\begin{lem}\label{lem:Kn-prim}
For every $n\geq0$, we have
\[
\Delta_*(R_n)
=
R_n\otimes1+1\otimes R_n.
\]
\end{lem}

\begin{proof}
We have already seen that $\Delta_*(R_0) = R_0 \otimes 1 + 1 \otimes R_0$.
Suppose that $\Delta_*(R_n) = R_n \otimes 1 + 1 \otimes R_n$ holds. Since $D$ is a coderivation,
\begin{align*}
\Delta_*(D(R_n))
&\stackrel{\eqref{eq:D-coderivation}}{=}
(D\otimes\id+\id\otimes D)
\Delta_*(R_n)\\
&\hspace{1ex}=
D(R_n)\otimes1
+
1\otimes D(R_n).
\end{align*}
By \eqref{eq:reccurent},
\[
R_{n+1}
=
-\frac{1}{n+1}D(R_n),
\]
and therefore $R_{n+1}$ is also so.
The assertion follows by induction.
\end{proof}

For $f\in\Q\langle Y\rangle$, put
\[
\widetilde{\Delta}_*(f)
:=
\Delta_*(f)-f\otimes1-1\otimes f.
\]
Since $\Delta_*(\pi_Y(\psi)-\Psi^{(0)}) = (\pi_Y(\psi)-\Psi^{(0)}) \otimes 1 + 1\otimes (\pi_Y(\psi)-\Psi^{(0)})$ holds, we have
\begin{align}
\widetilde{\Delta}_*(\pi_Y(\psi))
=
\widetilde{\Delta}_*(\Psi^{(0)}).
\label{eq:reduced-phi-Phi0}
\end{align}

\begin{lem}\label{lem:reduced-com}
For $\Psi=[x_1,\psi]\in\addmr_0$, we have
\begin{align}
\begin{aligned}
&
[\Delta_*(y_1),\widetilde{\Delta}_*(\pi_Y(\psi))]
\\
={}&
-\sum_{n\ge 0} \left(R_n + \sum_{i\ge 1} R_{n + i}y_i \right) \otimes y_{n + 1}\\
&-\sum_{n\ge 0} y_{n + 1} \otimes \left(R_n + \sum_{i\ge 1} R_{n + i}y_i \right) \\
&-
\sum_{n\geq1}
\sum_{\substack{i+j=n+1\\ i,j\geq1}}
\left(
[y_i,\Psi^{(n)}]\otimes y_j
+
y_i\otimes[y_j,\Psi^{(n)}]
\right).
\end{aligned} \label{eq:reduced-com}
\end{align}
\end{lem}

\begin{proof}
By Lemma~\ref{lem:Phi-com}, we have
\[
\sum_{n\geq0}[y_{n+1},\Psi^{(n)}]
=
-\sum_{n\geq0}R_ny_{n+1}.
\]
Applying $\widetilde{\Delta}_*$ to \eqref{eq:Phi-com}, we obtain
\begin{align*}
\sum_{n\geq0}
\widetilde{\Delta}_*
([y_{n+1},\Psi^{(n)}])
=
-\sum_{n\geq0}
\widetilde{\Delta}_*(R_ny_{n+1}).
\end{align*}

For $n=0$, since $\Delta_*(y_1)=y_1\otimes1+1\otimes y_1,$
we have
\begin{align*}
\widetilde{\Delta}_*([y_1,\Psi^{(0)}])
&\hspace{1ex}=
[\Delta_*(y_1),
\widetilde{\Delta}_*(\Psi^{(0)})]\\
&\stackrel{\eqref{eq:reduced-phi-Phi0}}{=}
[\Delta_*(y_1),
\widetilde{\Delta}_*(\pi_Y(\psi))].
\end{align*}

For $n\geq1$, since
\begin{align*}
\Delta_*(\Psi^{(n)})
=&
\Psi^{(n)}\otimes1+1\otimes\Psi^{(n)},\\
\Delta_*(y_{n+1})
=&
y_{n+1}\otimes1
+
1\otimes y_{n+1}
+
\sum_{\substack{i+j=n+1\\ i,j\geq1}}
y_i\otimes y_j
\end{align*}
holds, we have
\begin{align*}
\widetilde{\Delta}_*([y_{n+1}, \Psi^{(n)}]) =&[\widetilde{\Delta}_*(y_{n+1}), \Delta_*(\Psi^{(n)})] \\
=&\sum_{\substack{i + j = n + 1 \\ i,j >0}} [y_i \otimes y_j, \Psi^{(n)} \otimes 1 + 1\otimes \Psi^{(n)}]\\
=&\sum_{\substack{i + j = n + 1 \\ i,j >0}} [y_i , \Psi^{(n)}] \otimes y_j + y_i \otimes [y_j , \Psi^{(n)}].
\end{align*}

On the other hand,
\[
\widetilde{\Delta}_*(R_ny_{n+1})
=
\Delta_*(R_ny_{n+1})
-
R_ny_{n+1}\otimes1
-
1\otimes R_ny_{n+1}.
\]
Since, by Lemma~\ref{lem:Kn-prim}, $\Delta_*(R_n) = R_n \otimes 1 + 1\otimes R_n$ holds, we have
\begin{align*}
&
\Delta_*(R_ny_{n+1})
-
R_ny_{n+1}\otimes1
-
1\otimes R_ny_{n+1}
\\
={}&
R_n\otimes y_{n+1}
+
y_{n+1}\otimes R_n
\\
&+
\sum_{\substack{i+j=n+1\\ i,j\geq1}}
\left(
R_ny_i\otimes y_j
+
y_i\otimes R_ny_j
\right).
\end{align*}
Hence,
\begin{align*}
&
\sum_{n\geq0}
\left(
\Delta_*(R_ny_{n+1})
-
R_ny_{n+1}\otimes1
-
1\otimes R_ny_{n+1}
\right)
\\
={}&
\sum_{n\geq0}
\left(
R_n+\sum_{i\geq1}R_{n+i}y_i
\right)\otimes y_{n+1}
\\
&+
\sum_{n\geq0}
y_{n+1}\otimes
\left(
R_n+\sum_{i\geq1}R_{n+i}y_i
\right).
\end{align*}
Combining these identities gives the assertion.
\end{proof}

\begin{lem}\label{lem:com-red}
For $\Psi=[x_1,\psi]\in\Fadred$ and $k\in\mathbb Z_{\ge0}$, we have
\begin{align}
\Psi^{(k)}
+\sum_{i\ge1}y_i\Psi^{(k+i)}
=
\Psi_{(k)}
+\sum_{i\ge1}\Psi_{(k+i)}y_i.
\label{eq:com-red}
\end{align}
\end{lem}

\begin{proof}
We first consider the case $k=0$.
Since $\Psi\in\Fadred$, we have
\[
\Psi^{11}+\Psi^{10}+\Psi^{01}=0.
\]

On the other hand, by definition of $\Psi^{(n)}$ and $\Psi_{(n)}$, we have
\begin{align*}
\pi_Y(\Psi^{11}x_1) = \Psi_{(0)} - \Psi^{(0)},\quad \pi_Y(\Psi^{01}x_1) = - \sum_{i\ge 1}y_i \Psi^{(i)}, \quad \pi_Y(\Psi^{10}x_1) = \sum_{i\ge 1} \Psi_{(i)}y_i.
\end{align*}
These imply the case for $k = 0$. 

Assume that \eqref{eq:com-red} holds for some $k\ge 0$. Since
\begin{align*}
0 = D(\pi_Y(\psi x_1)) = \sum_{n\ge 0} ny_{n}\Psi^{(n)} + y_{n+1}D(\Psi^{(n)})\\
0 = D(\pi_Y(\psi x_1)) = \sum_{n\ge 0} n\Psi_{(n)}y_n + D(\Psi_{(n)})y_{n+1}
\end{align*}
holds, we have
\begin{align*}
D(\Psi^{(k)}) =& -(k + 1) \Psi^{(k + 1)}\\
D(\Psi_{(k)}) =& -(k + 1) \Psi_{(k + 1)}
\end{align*}
Thus, applying D to \eqref{eq:com-red}, we have
\begin{align*}
D\text{(LHS of  \eqref{eq:com-red})} =&
D\left(\Psi^{(k)} + \sum_{i\ge 1} y_i \Psi^{(k+i)}\right)\\
=& - (k+1)\Psi^{(k+1)} + \sum_{i\ge 1} (i - 1)y_{i-1} \Psi^{(k + i)} - \sum_{i\ge 1} (k + i + 1)y_{i} \Psi^{(k + i + 1)}\\
=&-( k +1)\left(\Psi^{(k+1)} + \sum_{i\ge 1} y_{i} \Psi^{(k + i + 1)} \right)
\end{align*}
and
\begin{align*}
D\text{(RHS of \eqref{eq:com-red})}
&=D\left(\Psi_{(k)}+\sum_{i\ge1}\Psi_{(k+i)}y_i\right)\\
&=-(k+1)\Psi_{(k+1)}
+\sum_{i\ge2}(i-1)\Psi_{(k+i)}y_{i-1}
-\sum_{i\ge1}(k+i+1)\Psi_{(k+i+1)}y_i\\
&=-(k+1)\left(
\Psi_{(k+1)}+\sum_{i\ge1}\Psi_{(k+i+1)}y_i
\right).
\end{align*}
Dividing by $-(k+1)$, we obtain \eqref{eq:com-red} for $k+1$.
\end{proof}

\begin{lem}\label{lem:y1y1}
For $\Psi = [x_1,\psi] \in \addmr_0 \cap \Fadred$, we have
\begin{align*}
[\Delta_*(y_1), \widetilde{\Delta}_*(\pi_Y(\psi))] = 0. \label{eq:y1y1}
\end{align*}
\end{lem}

\begin{proof}
By the definition of $R_n$, we have
\[
R_n=\Psi^{(n)}-\Psi_{(n)}.
\]
Thus,
\begin{align*}
R_k+\sum_{i\geq1}R_{k+i}y_i
={}&
\Psi^{(k)}
+\sum_{i\geq1}\Psi^{(k+i)}y_i
-
\Psi_{(k)}
-
\sum_{i\geq1}\Psi_{(k+i)}y_i.
\end{align*}
By Lemma~\ref{lem:com-red},
\[
\Psi_{(k)}
+
\sum_{i\geq1}\Psi_{(k+i)}y_i
=
\Psi^{(k)}
+
\sum_{i\geq1}y_i\Psi^{(k+i)}.
\]
Therefore,
\begin{align*}
R_k+\sum_{i\geq1}R_{k+i}y_i
&=
\sum_{i\geq1}
\left(
\Psi^{(k+i)}y_i-y_i\Psi^{(k+i)}
\right)\\
&=
-\sum_{i\geq1}[y_i,\Psi^{(k+i)}].
\end{align*}
It follows that
\begin{align*}
&
\sum_{n\ge 0} \left(R_n + \sum_{i\ge 1} R_{n + i}y_i \right) \otimes y_{n + 1} +\sum_{n\ge 0} y_{n + 1} \otimes \left(R_n + \sum_{i\ge 1} R_{n + i}y_i \right) \\
={}&
-\sum_{k\geq0}\sum_{i\geq1}
[y_i,\Psi^{(k+i)}]\otimes y_{k+1}
-
\sum_{k\geq0}\sum_{i\geq1}
y_{k+1}\otimes[y_i,\Psi^{(k+i)}]
\\
={}&
-\sum_{n\geq1}
\sum_{\substack{i+j=n+1\\i,j\geq1}}
\left(
[y_i,\Psi^{(n)}]\otimes y_j
+
y_i\otimes[y_j,\Psi^{(n)}]
\right),
\end{align*}
where in the last equality we reindex by $n=k+i$.

Consequently, the right hand side of \eqref{eq:reduced-com} is canceled, and we conclude
\[
[\Delta_*(y_1),\widetilde{\Delta}_*(\pi_Y(\psi))]=0.
\]
\end{proof}

\begin{lem}\label{lem:kery1y1}
We have
\[
\ker \ad(\Delta_*(y_1)) = \Q\langle y_1 \rangle^{\otimes 2}.
\]
\end{lem}

\begin{proof}
Let $W\subset\Q\langle Y\rangle$ be the span of the words containing
a letter other than $y_1$. Since
$\ker\ad(y_1)=\Q\langle y_1\rangle$ and $\ad(y_1)(W)\subset W$,
the restriction $\ad(y_1)|_W$ is injective.
Moreover,
\[
\ad(\Delta_*(y_1))
=\ad(y_1)\otimes\id+\id\otimes\ad(y_1)
\]
preserves the decomposition
\[
\Q\langle Y\rangle^{\otimes 2}
=\Q\langle y_1\rangle^{\otimes 2}
\oplus(\Q\langle y_1\rangle\otimes W)
\oplus(W\otimes\Q\langle y_1\rangle)
\oplus W^{\otimes 2}.
\]
It vanishes on the first summand and is injective on the two mixed
summands by the injectivity of $\ad(y_1)|_W$.

For a nonzero $w\in W^{\otimes 2}$, let $m$ be the largest weight
occurring in the left factor, and let $w_m\ne0$ be the corresponding
component. The component of $\ad(\Delta_*(y_1))(w)$ of left weight
$m+1$ is $(\ad(y_1)\otimes\id)(w_m)$, which is nonzero.
Thus the restriction to $W^{\otimes 2}$ is also injective,
proving the assertion.
\end{proof}

\begin{proof}[Proof of the surjectivity of $\ad(x_1) : \dmr_0 \rightarrow \addmr_0 \cap \Fadred$]

Conditions \eqref{df:ds-1}--\eqref{df:ds-3} of Definition~\ref{df:ds} for $\psi$ follow from the choice of $\psi$ and Definition~\ref{df:addmr}~(1).
It is well known (cf.~\cite[\S~4]{Hoffman}) that there exists a unique commutative bilinear binary operation $*$ on $\Q\langle Y \rangle$ such that
\begin{align}
\langle \Delta_*(f) \mid u \otimes v \rangle = \langle f \mid u * v \rangle \label{eq:harmonic-du}
\end{align}
holds for $f \in \Q\langle Y \rangle$ and $u, v \in Y^*\setminus\{1\}$.
Since we have $\widetilde{\Delta}_*(\pi_Y(\psi)) \in \Q\langle y_1 \rangle^{\otimes 2}$ by Lemma~\ref{lem:y1y1} and Lemma~\ref{lem:kery1y1}, it follows that
\begin{align}
0 = \langle \Delta_*(\pi_Y(\psi)) \mid u\otimes v \rangle = \langle \pi_Y(\psi) \mid u * v\rangle \label{eq:harmonic}
\end{align}
for $u, v\in Y^*\setminus\{1\}$ unless both $u$ and $v$ are powers of $y_1$.
Thus, to prove $\psi \in \dmr_0$, it suffices to show that
\[
\langle\Delta_*(\psi_{*}) \mid y_1^m \otimes y_1^n \rangle = 0
\]
for all $m, n \in \mathbb{Z}_{>0}$.

By the coassociativity of $\Delta_*$, we have
\begin{align*}
&n\langle \Delta_*(\pi_Y(\psi)) \mid y_1^{m} \otimes y_1^n \rangle\\
=&\langle (\Delta_* \otimes \id) \circ \Delta_*(\pi_Y(\psi)) \mid y_1^{m} \otimes y_1^{n-1} \otimes y_1 \rangle\\
=&\langle (\id \otimes \Delta_*) \circ \Delta_*(\pi_Y(\psi)) \mid y_1^{m} \otimes y_1^{n-1} \otimes y_1 \rangle\\
=&\binom{m + n -1}{m}\langle \Delta_*(\pi_Y(\psi)) \mid y_1^{m + n - 1} \otimes y_1\rangle.
\end{align*}
Here, we use Lemma~\ref{lem:y1y1} and \ref{lem:kery1y1} for calculating first and third equalities.
Therefore, it suffices to determine
\[
\langle \Delta_*(\pi_Y(\psi)) \mid y_1^{l} \otimes y_1\rangle
\]
for $l \in \mathbb{Z}_{>0}$.
By the definition of $*$ (cf.~\cite[\S2]{Hoffman}) and \eqref{eq:harmonic}, we have, for $m \in \mathbb{Z}_{\ge 2}$,
\[
0 = \langle\pi_Y(\psi) \mid y_1^l * y_m \rangle = \sum_{ i+ j = l} \langle \pi_Y(\psi) \mid y_1^i y_m y_1^j \rangle + \sum_{ i+ j = l-1} \langle \pi_Y(\psi) \mid y_1^i y_{m+1} y_1^j \rangle,
\]
that is,
\begin{align}
 \sum_{ i+ j = l} \langle \pi_Y(\psi) \mid y_1^i y_m y_1^j \rangle = - \sum_{ i+ j = l-1} \langle \pi_Y(\psi) \mid y_1^i y_{m+1} y_1^j \rangle \label{eq:har-iter}
\end{align}
for $m \in \mathbb{Z}_{\ge 2}$.
By iterating \eqref{eq:har-iter}, we obtain
\begin{align*}
& \langle \Delta_*(\pi_Y(\psi)) \mid y_1^{l} \otimes y_1\rangle \\
\stackrel{\eqref{eq:harmonic-du}}{=}  & \langle \pi_Y(\psi) \mid y_1^l * y_1 \rangle\\
=\hspace{0.2cm}&(l + 1)\langle \pi_Y(\psi) \mid y_1^{l+1} \rangle + \sum_{i + j = l - 1} \langle \pi_Y(\psi) \mid y_1^i y_2 y_1^j \rangle\\
=\hspace{0.2cm}&\sum_{i + j = l-1} \langle \pi_Y(\psi) \mid y_1^i y_2 y_1^j \rangle\\
\stackrel{\eqref{eq:har-iter}}{=}& -\sum_{i + j = l-2} \langle \pi_Y(\psi) \mid y_1^i y_3 y_1^j \rangle \\
\vdots&\\
=\hspace{0.2cm}&(-1)^{l+1} \langle \pi_Y(\psi) \mid y_{l+1} \rangle.
\end{align*}
Here, in the third equality, we use
\[
\langle \pi_Y(\psi) \mid y_1^{l +1} \rangle = \langle \psi \mid x_1^{l+1} \rangle = 0,
\]
which comes from the property of $\psi \in \ide{F}_2$ and $\langle \psi \mid x_1 \rangle = 0$.

On the other hand, we have
\begin{align*}
\langle \Delta_*(\psi_{\corr}) \mid y_1^l \otimes y_1 \rangle =&\left\langle \Delta_*\left(\sum_{m\ge 2} \frac{(-1)^{m-1}}{m} \langle \pi_Y(\psi) \mid y_m \rangle y_1^m \right) \middle| y_1^l \otimes y_1 \right\rangle\\
=& \frac{(-1)^l}{l+1} \binom{l + 1}{1} \langle \pi_Y(\psi) \mid y_{l + 1} \rangle\\
=&(-1)^{l}  \langle \pi_Y(\psi) \mid y_{l + 1} \rangle.
\end{align*}
Thus, by the coassociativity again, we have, for $m, n \in \mathbb{Z}_{>0}$, 
\begin{align*}
&\langle \Delta_*(\pi_Y(\psi)) +  \Delta_*(\psi_{\corr}) \mid y_1^m \otimes y_1^n \rangle \\
=&\binom{m + n-1}{m}n^{-1} \left(\langle  \Delta_*(\pi_Y(\psi)) +  \Delta_*(\psi_{\corr}) \mid y_1^{m + n -1} \otimes y_1 \rangle\right)\\
=&\binom{m + n-1}{m}n^{-1}\left( (-1)^{m + n} \langle \pi_Y(\psi) \mid y_{m + n} \rangle + (-1)^{m + n -1}  \langle \pi_Y(\psi) \mid y_{m + n} \rangle \right)\\
=&0.
\end{align*}
This proves the surjectivity of $\ad(x_1) : \dmr_0 \rightarrow \addmr_0 \cap \Fadred$.
\end{proof}

\begin{proof}[Proof of $\ad(x_1)(\dmr_0) \subset \Fadred$]
Let $\psi\in\dmr_0$ and put $\Psi=[x_1,\psi]$.
Since $\Psi\in\addmr_0$ by \cite[Proposition~3.2.5~(ii)]{Jarossay},
it suffices to show that $\Psi^{11}+\Psi^{10}+\Psi^{01}=0$.
By Lemma~\ref{lem:Kn-prim}, we have
\[
\Delta_*(R_n)=R_n\otimes1+1\otimes R_n \qquad (n\ge0).
\]

Since $\Delta_*(\psi_*)=\psi_*\otimes1+1\otimes\psi_*$, we have
$\widetilde{\Delta}_*(\pi_Y(\psi))=-\widetilde{\Delta}_*(\psi_{\corr})$,
which lies in $\Q\langle y_1\rangle^{\otimes2}$ and hence commutes with $\Delta_*(y_1)$.
Thus the left-hand side of \eqref{eq:reduced-com} vanishes.
On the other hand, substitute $R_n=\Psi^{(n)}-\Psi_{(n)}$ into the right-hand side of \eqref{eq:reduced-com}, and reindex its last sum as in the proof of Lemma~\ref{lem:y1y1}.
Then the vanishing of the right-hand side of \eqref{eq:reduced-com} yields
\begin{align}
\begin{aligned}
&\sum_{n\ge 0} \left(\Psi_{(n)} + \sum_{i\ge 1} \Psi_{(n + i)} y_i  \right) \otimes y_{n + 1} + \sum_{n\ge 0} y_{n + 1} \otimes \left(\Psi_{(n)} + \sum_{i\ge 1} \Psi_{(n + i)} y_i  \right)\\
=&\sum_{n\ge 0} \left(\Psi^{(n)} + \sum_{i\ge 1} y_i \Psi^{(n + i)}  \right) \otimes y_{n + 1} + \sum_{n\ge 0} y_{n + 1} \otimes \left(\Psi^{(n)} + \sum_{i\ge 1} y_i\Psi^{(n + i)}  \right).
\end{aligned}
 \label{eq:comp-tensor}
\end{align} 
Comparing $\mathchar`-\otimes y_{1}$ of \eqref{eq:comp-tensor} imply
\begin{align*}
\Psi_{(0)} + \sum_{i\ge 1} \Psi_{(i)} y_i + \sum_{n\ge 0}  \langle \Psi_{(n)} \mid y_1 \rangle y_{n+1} 
=\Psi^{(0)} + \sum_{i\ge 1} y_i \Psi^{(i)} + \sum_{n\ge 0}  \langle \Psi^{(n)} \mid y_1 \rangle y_{n + 1} .
\end{align*}
Here, since $\psi \in \ide{F}_2$, $S_X(\psi) = - \psi$ holds. Therefore, we have
\[
\langle \Psi^{(n)} \mid y_1 \rangle = \langle \psi \mid x_0^nx_1 \rangle = - \langle S_X(\psi) \mid x_0^nx_1 \rangle = (-1)^{n} \langle \psi \mid x_1x_0^n \rangle = (-1)^n \langle \Psi_{(n)} \mid y_1 \rangle .
\]
If $n$ is even, we have $\langle \Psi^{(n)} \mid y_1 \rangle = \langle \Psi_{(n)} \mid y_1 \rangle$. If $n$ is odd, then $\langle \psi \mid x_0^nx_1 \rangle = 0$ by \cite[Proposition~3.3.3]{Racinet}, and hence both sides vanish. These imply
\[
\sum_{n\ge 0}  \langle \Psi_{(n)} \mid y_1 \rangle y_{n+1}  = \sum_{n\ge 0}  \langle \Psi^{(n)} \mid y_1 \rangle y_{n + 1},
\]
and we have
\[
\Psi_{(0)} + \sum_{i\ge 1} \Psi_{(i)} y_i 
=\Psi^{(0)} + \sum_{i\ge 1} y_i \Psi^{(i)}.
\]
Since we already checked
\begin{align*}
\pi_Y(\Psi^{11}x_1) = \Psi_{(0)} - \Psi^{(0)},\quad \pi_Y(\Psi^{01}x_1) = - \sum_{i\ge 1}y_i \Psi^{(i)}, \quad \pi_Y(\Psi^{10}x_1) = \sum_{i\ge 1} \Psi_{(i)}y_i,
\end{align*}
we have
\[
\pi_Y(\Psi^{11}x_1 + \Psi^{10}x_1 + \Psi^{01}x_1) = 0.
\]

Since $\Q\langle X \rangle \ni f \mapsto\pi_Y(fx_1) \in \Q\langle Y \rangle$ is injective, we have 
\[
\Psi^{11} + \Psi^{10} + \Psi^{01} = 0.
\]
and this concludes $\Psi \in \Fadred$.
\end{proof}

\subsection*{Conclusion}

By considering the natural completion, $\widehat{\addmr_0} \cap \cFadred$ is a pronilpotent Lie algebra and its scalar extension
\[
R\cotimes \widehat{\addmr_0} \cap \cFadred
\]
is also a Lie algebra. Thus, by the exponential map,
\[
\exp^{\cast_1}(R\cotimes \widehat{\addmr_0}\cap\cFadred)
\]
is a prounipotent group equipped with the binary operation $\cast_1$ and the following diagram commutes:
\[
\xymatrix{
\DMR_0(R) \ar[r]^-{\Ad(x_1)} & \AdDMR_0(R) \cap \exp(\cFadred)(R) \\
\exp^{\cast}(R\cotimes \widehat{\dmr}_0) \ar[r]^-{\Ad(x_1)}_-{\sim} \ar@{=}[u] & \exp^{\cast_1}(R\cotimes \widehat{\addmr}_0 \cap \cFadred) \ar@{^(-^>}[u]\\
R\cotimes \widehat{\dmr}_0 \ar[u]^{\exp^{\cast}}_-{\rotatebox{90}{$\sim$}} \ar[r]_-{\ad(x_1)}^-{\sim} & R\cotimes \widehat{\addmr}_0 \cap \cFadred \ar[u]_{\exp^{\cast_1}}^-{\rotatebox{90}{$\sim$}}
}
\]
Here, the bottom square follows from Proposition~\ref{prop:ism} and the naturality of $\exp^{\cast}$ and $\exp^{\cast_1}$.
Thus, to prove Theorem~\ref{thm:ism}, it suffices to show
\[
\AdDMR_0(R) \cap \exp(\cFadred)(R)  \subset  \exp^{\cast_1}(R\cotimes (\widehat{\addmr}_0 \cap \cFadred)).
\]
 Let us take $\Phi \in \AdDMR_0(R) \cap \exp(\cFadred)(R)$. By \eqref{eq:exp-ad}, there exists $\Psi \in R\cotimes \cFadred$ such that
\[
\Phi = \exp^{\cast_1}(\Psi).
\]
By definition, it immediately verifies that
\[
\Psi \in R\cotimes \cFadred.
\]
and $\langle \Psi \mid w \rangle = 0$ for $\wt w = 1$, $2$, $3$.
Thus, it suffices to show
\[
\Delta_* (\Psi_{\#}) = \Psi_{\#} \otimes 1 + 1 \otimes \Psi_{\#}.
\]

Let
\[
\Psi = \sum_{n\ge 0} \Psi^{[n]}
\]
be the weight-homogeneous decomposition of $\Psi$ and put
\[
\Psi_n := \sum_{0\le m < n} \Psi^{[m]}.
\]
Since $\Psi^{[k]}  = 0$ for $1\le k \le 3$, we have $\Psi^{[k]} \in R\cotimes (\addmr_0 \cap \Fadred)$.
Assume that $\Psi^{[k]} \in R\cotimes (\addmr_0 \cap \Fadred)$ for positive integers $k$ less than $n \in \mathbb{Z}_{>3}$.
Since
\[
\exp^{\cast_1}(R\cotimes (\addmr_0 \cap \Fadred)) \subset \AdDMR_0 \cap \exp(\Fadred)(R)
\]
holds, we have
\[
\exp^{\cast_1}(\Psi_n) \in \AdDMR_0\cap \exp(\cFadred)(R).
\]
Since
\[
\exp^{\cast_1}(\Psi)_{\#} - \exp^{\cast_1}(\Psi_n)_{\#} = \Psi_{\#}^{[n]} + O(X^{n}),
\]
we have
\begin{align*}
\Delta_*(\exp^{\cast_1}(\Psi)_{\#} - \exp^{\cast_1}(\Psi_n)_{\#})
=&\exp^{\cast_1}(\Psi)_{\#}  \otimes \exp^{\cast_1}(\Psi)_{\#}  -  \exp^{\cast_1}(\Psi_n)_{\#} \otimes  \exp^{\cast_1}(\Psi_n)_{\#}\\
=&\Psi^{[n]}_{\#} \otimes 1 + 1\otimes \Psi^{[n]}_{\#} + O(X^{n}).
\end{align*}
Here, $O(X^{n})$ is the infinite linear combination of terms whose weight is larger than $n-1$ (note that $\pi_{Y,\#}$ maps weight $n$ terms of $\Q\langle\langle X\rangle\rangle$ to weight $n-1$ terms of $\Q[[T;Y\rangle\rangle$, so $\wt(\psi_{\#}^{[n]}) = n-1$).
Since the weight $n-1$-homogeneous part of the above are equal, thus we have
\[
\Delta_*(\Psi^{[n]}_{\#}) = \Psi^{[n]}_{\#} \otimes 1 + 1 \otimes \Psi^{[n]}_{\#}
\]
and get our conclusion.

\bibliographystyle{amsplain}

\bibliography{reference-ad-red}

\end{document}